\documentclass{amsart}

\usepackage[all]{xy}    
\usepackage{graphicx}
\usepackage{amssymb}
\usepackage{mathrsfs}
\usepackage{multirow}
\usepackage{float}
\usepackage{tikz-cd}
\usepackage{adjustbox}
\usepackage{amsthm}
\usepackage{accents}
\usepackage{mathtools}
\usepackage{enumitem,cite,array}
\usepackage[new]{old-arrows}
\usepackage{tikz}
\usetikzlibrary{matrix,arrows}
\usepackage{amsmath}

\usepackage[numbers,sort&compress]{natbib} 
\usepackage[bookmarksnumbered, bookmarksopen,
colorlinks,citecolor=blue,linkcolor=blue,backref]{hyperref}

\newcounter{RomanNumber}

\newtheorem{theorem}{Theorem}
\newtheorem{lemma}[theorem]{Lemma}

\theoremstyle{definition}

\newtheorem{example}[theorem]{Example}
\newtheorem{proposition}[theorem]{Proposition}

\theoremstyle{notation}

\newcommand{\be}{\begin{equation}}
\newcommand{\ee}{\end{equation}}

\usepackage[utf8]{inputenc}
\usepackage{chngcntr}
\usepackage{apptools}
\AtAppendix{\counterwithin{theorem}{section}}
\usepackage[toc,page]{appendix}

\newcommand{\qqed}{\hfill\Box}

\begin{document}


\title
{Coformality around fibrations and cofibrations}


 \author{Ruizhi Huang}
\address{Institute of Mathematics and Systems Sciences, Chinese Academy of Sciences, Beijing 100190, China}
\email{huangrz@amss.ac.cn}  
\urladdr{https://sites.google.com/site/hrzsea
}
\thanks{}

\date{}

\maketitle

\begin{abstract}
We show that in a fibration the coformality of the base space implies the coformality of the total space under reasonable conditions, and these conditions can not be weakened. The result is partially dual to the classical work of Lupton \cite{Lup} on the formality within a fibration. Our result has two applications. First, we show that for certain cofibrations, the coformality of the cofiber implies the coformality of the base. Secondly, we show that the total spaces of certain spherical fibrations are Koszul in the sense of Berglund \cite{Ber}.
\end{abstract}

\tableofcontents


\section*{Introduction}
\subsection{Background}
\label{subsec: backgroundintro}
Let $X$ be a simply connected space of finite type. $X$ is rationally {\it formal} if its rational homotopy type is determined by the graded commutative algebra $H^\ast(X;\mathbb{Q})$; and is rationally {\it coformal} if its rational homotopy type is determined by the graded Lie algebra $\pi_\ast(\Omega X)\otimes \mathbb{Q}$. Formality and coformality are important properties in rational homotopy theory. The classical work of Deligne, Griffiths, Morgan and Sullivan \cite{DGMS} shows that compact K\"{a}hler manifolds are formal. Coformality was first introduced by Neisendorfer and Miller \cite{NM} as a dual concept of formality. Both of them provides building blocks for rational homotopy types from different points of view. Recently, Berglund \cite{Ber} combined them together to study Koszul structure in rational homotopy theory.

To characterize the formality of spaces in a homotopy fibration is a classical problem. Lupton \cite{Lup}, and Amann and Kapovitch \cite{AK} proved deep results on the relation between the formality of the total space  of a fibration and that of the base space under reasonable conditions. In this paper, we study the same question for coformality, that is, to investigate the coformality of the base space and that of the total space in a fibration.

\subsection{Coformality}
\label{subsec: coformalintro}
When studying formality around fibrations, it is common to impose a so-called TNCZ condition on the fibration, for instance in \cite[Proposition 3.2]{Lup} and \cite[Theorem A, Corollary B]{AK}. To be precise, a fibration $F\stackrel{i}{\hookrightarrow} E\stackrel{}{\rightarrow} B$ is called {\it rationally totally non-cohomologous to zero (TNCZ)}, if the induced homomorphism $i^\ast: H^\ast(E;\mathbb{Q})\rightarrow H^\ast(F;\mathbb{Q})$ is surjective. In order to study coformality around fibration, we may wish to impose a dual condition of TNCZ. For this, let us call the fibration is {\it rationally totally non-homotopical to zero (TNHZ)}, if the induced homomorphism $i_\ast: \pi_\ast(F)\otimes \mathbb{Q}\rightarrow \pi_\ast(E)\otimes \mathbb{Q}$ is injective. The condition TNCZ is equivalent to that the $E_2$-term of the Serre spectral sequence of the fibration collapses, while the condition TNHZ is equivalent to that the long exact sequence of the homotopy groups of the fibration splits rationally.

To study coformality, we introduces a notion of coformal limit associated to any simply connected space. More precisely, let $X$ be a simply connected topological space with a minimal Sullivan model $(\Lambda V, d)$. We define the {\it coformal limit} $X^\prime$ of $X$ to be the geometric realization of the purely quadratic Sullivan algebra $(\Lambda V, d_1)$ with the differential $d_1$ being the quadratic part of $d$. 
In Section \ref{sec: limit}, we show that coformal limit is well defined up to rational homotopy equivalence and is a covariant functor. Moreover, a space $X$ is coformal if and only if it is rational homotopy equivalent to its limit $X^\prime$. This justifies the terminology.

We can now state our first theorem in this paper. Denote by $cat_0(X)$ the {\it rational LS category} of a space $X$.
Recall $X$ is called of {\it finite type} if $H_i(X;\mathbb{Q})$ is finite dimensional for any $i\in \mathbb{Z}$. 
\begin{theorem}\label{coformalmainthm}
Let 
\[
F\stackrel{}{\longhookrightarrow} E\stackrel{}{\longrightarrow} B
\]
be a fibration of simply connected topological spaces of finite type. Suppose the fibration is TNHZ, and $F$ is  coformal. If $B$ is coformal and $cat_0(E^\prime)\leq 2$, then $E$ is coformal. 
\end{theorem}
Lupton \cite[Proposition 3.2]{Lup} showed that if a fibration is TNCZ and the fiber is formal and elliptic, then the formality of the base implies the formality of the total space. Despite the elliptic and category conditions, Theorem \ref{coformalmainthm} is dual to the result of Lupton. 

Under further conditions on the fiber, Amann and Kapovitch \cite{AK} showed the formality of the total space implies the formality of the base. However, this is not case for coformality. Indeed, in Example \ref{counterex2} we constructed a fibration $S^3\rightarrow E\rightarrow \mathbb{C}P^2$ such that $E$ is coformal with $cat_0(E^\prime)=2$. But since $\mathbb{C}P^2$ is not coformal, this disproves the guess that the coformality of the total space implies the coformality of the base. 

Nevertheless, the coformality of the total space always implies the coformality of the fiber as proved in Proposition \ref{coformalprop3}. In addition, the TNHZ condition is necessary as illustrated by a fibration $S^2\rightarrow \mathbb{C}P^3\rightarrow S^4$ in Example \ref{counterex1}. Moreover, the category inequality condition is optimal. In Example \ref{counterex3}, we construct a rational TNHZ fibration $S^9\rightarrow E\rightarrow S^2\times S^2$ such that $E$ is not coformal and $cat_0(E^\prime)=3$.

The following is an example of Theorem \ref{coformalmainthm}.
\begin{example}\label{coformalfibex1intro}
Let $F\stackrel{i}{\hookrightarrow} E\stackrel{}{\rightarrow} B$
be a fibration where $F$ is an odd dimensional sphere and $B$ is a simply connected wedge of spheres of finite type. It is clear that $B$ and $F$ are coformal. Moreover, it is easy to see that the coformal limit functor gives a firation $F\rightarrow E^\prime\rightarrow B$ (\ref{E'fibeq}), and then by \cite[Proposition 30.7 (ii)]{FHT} $cat_0(E^\prime)\leq cat_0(B)+cat_0(F)=2$. Hence if $i_\ast$ is rationally nontrivial, then $E$ is coformal. 
\end{example}

Theorem \ref{coformalmainthm} can be applied to study coformality around cofibration. 
Let $\Sigma A\stackrel{f}{\longrightarrow} Y \stackrel{h}{\longrightarrow} Z$ be a homotopy cofibration. The map $f$ is called {\it inert} if $h$ is surjective in rational homotopy. Notice that when $A=S^{k}$ for some $k$, this is the classical notion of inert map. We can now state our second theorem.
\begin{theorem}\label{coformalmainthm2}
Let
\[
\Sigma A\stackrel{f}{\longrightarrow} Y \stackrel{}{\longrightarrow} Z
\]
be a homotopy cofibration of simply-connected spaces of finite type such that $f$ is inert. If $Z$ is coformal and $cat_0(Y^\prime)\leq 2$, then $Y$ is coformal.
\end{theorem}
It is known that a connected sum of products of two simply connected spheres is coformal. This is an example of Theorem \ref{coformalmainthm2} discussed in Example \ref{coformalfibex4}. Here we provide two more examples.

\begin{example}\label{coformalfibex2intro}
Let $M$ be a closed simply connected $n$-manifold whose rational cohomology algebra is not generated by one single class. Let $Y=M-\{{\rm pt}\}$ be the {\it deleted manifold}. There is the canonical homotopy cofibration
\[
S^{n-1}\stackrel{f}{\longrightarrow} Y\stackrel{}{\longrightarrow} M,
\]
where $f$ is the attaching map of the top cell of $M$.
By \cite{HL} it is known that $f$ is inert. Hence if $M$ is coformal and $cat_0(Y^\prime)\leq 2$, then the deleted manifold $Y$ is coformal.
\end{example}

\begin{example}\label{coformalfibex3intro}
For $n\geq 3$ and $n\neq 4, 8$, let $M$ be a $(n-1)$-connected closed $2n$-manifold such that $H^n(M;\mathbb{Z})\cong \mathbb{Z}^{d+2}$ with $d\geq 0$. In \cite[Example 4.2]{BT}, Beben-Theriault showed that there is a rational homotopy cofibration
\[
\bigvee_{i=1}^{d} S^n \stackrel{f}{\longrightarrow} M \stackrel{}{\longrightarrow} S^n\times S^n,
\]
such that $f$ is inert. Hence if $cat_0(M^\prime)\leq 2$, then $M$ is coformal.
\end{example}

\subsection{Koszul property}
\label{subsec: koszulintro}
Theorem \ref{coformalmainthm} can be applied to study the so called {\it Koszul property} introduced by Berglund \cite{Ber}. Indeed, he \cite{Ber} defined the notion of {\it Koszul space} as geometric realization of graded commutative {\it Koszul algebra} and showed that a space is Koszul if and only if it is both formal and coformal. 

Recall a simply connected topological space $X$ is {\it elliptic} if both $H^\ast(X;\mathbb{Q})$ and $\pi_\ast(X)\otimes \mathbb{Q}$ are finite dimensional.
In \cite[Proposition 3.2]{Lup}, Lupton showed that if a fibration is TNCZ and the fiber is formal and elliptic, then the formality of the base implies the formality of the total space. Combining this result with Theorem \ref{coformalmainthm}, we immediately obtain Proposition \ref{fibkoszulthm} for the Koszul property around fibrations. 
\begin{proposition}\label{fibkoszulthm}
Let 
\[
F\stackrel{}{\longhookrightarrow} E\stackrel{}{\longrightarrow} B
\]
be a fibration of simply connected topological spaces of finite type. Suppose that the fibration is TNHZ and TNCZ, and $F$ is Koszul, elliptic. If $B$ is Koszul and $cat_0(E^\prime)\leq 2$, then $E$ is Koszul. ~$\qqed$
\end{proposition}

The new conditions in Proposition \ref{fibkoszulthm} impose strong restriction on rational category. Indeed, 
by a theorem of Jessup \cite[Proposition 3.6]{Jes}, if the fibration $F\rightarrow E\rightarrow B$ is TNCZ with $F$ formal, then $cat_0(E)\geq cat_0(B)+cat_0(F)$. In particular, when $B$ and $F$ are both rationally noncontractible, $cat_0(E)\geq cat_0(B)+cat_0(F)\geq 2$. Suppose all the conditions of Proposition \ref{fibkoszulthm} hold for the fibration, then $E\simeq_{\mathbb{Q}} E^\prime$, and $cat_0(E)=2$, which implies that $cat_0(B)=cat_0(F)=1$. It implies that both $B$ and $F$ are rationally wedges of spheres by \cite[Theorem 28.5 (ii)]{FHT}, and further $F$ is a sphere since it is elliptic. By the above argument, Proposition \ref{fibkoszulthm} descents to the following theorem which will be proved in Section \ref{sec: proofmainthm}.

\begin{theorem}\label{fibkoszulthm2}
Let 
\[
S^n\stackrel{i}{\longhookrightarrow} E\stackrel{}{\longrightarrow} B
\]
be a fibration with $B$ a simply connected wedge of spheres of finite type. 
Suppose the fibration satisfies one of the following:
\begin{itemize}
\item[(1).] $n$ is odd and $i^\ast$ is rationally nontrivial on cohomology;
\item[(2).] $n$ is even, the fibration is TNHZ and TNCZ;
\item[(3).] $n$ is even, $B$ is a wedge of odd dimensional spheres, and the fibration is TNCZ.
\end{itemize}
Then $E$ is Koszul.
\end{theorem}

\subsection{Organization of the paper}
The paper is organized as follows. 
In Section \ref{sec: prelim}, we review some basics on rational homotopy theory for the latter use and to fix notations. We also prove two preliminary lemmas and show a simple case of coformality around fibration as warm up.
In Section \ref{sec: limit}, we introduce the notion of coformal limit and study its basic properties.
Section \ref{sec: proofmainthm} is devoted to prove the main theorems, Theorem \ref{coformalmainthm}, Theorem \ref{coformalmainthm2} and Theorem \ref{fibkoszulthm2}, of this paper.
In Section \ref{sec: counterexamples}, we show that in Theorem \ref{coformalmainthm} the condition on the coformality of the fiber is natural by Proposition \ref{coformalprop3}, the TNCZ condition is necessary and the category condition is optimal by counterexamples.

\bigskip

\noindent{\bf Acknowledgements.}
Ruizhi Huang was supported by National Natural Science Foundation of China (Grant nos. 11801544 and 11688101), and ``Chen Jingrun'' Future Star Program of AMSS.

\numberwithin{equation}{section}
\numberwithin{theorem}{section}
\section{Preliminary results} 
\label{sec: prelim} 
In this section, we first briefly recall some necessary terminologies and notations used in this paper, while for the detailed knowledge of rational homotopy theory one can refer to the standard literature \cite{FHT}. Then we show two preliminary lemmas which are useful to prove coformality in our context. As a warm up, we also prove one simple case of coformality around fibration with examples.

Recall a $CW$ complex $X$ is rationally {\it formal} if the commutative differentiable graded algebra $(H^\ast(X;\mathbb{Q}), 0)$ is a commutative model of $X$; and is rationally {\it coformal} if the differentiable graded Lie algebra $(L_X, 0)$ is a Lie model of $X$, where $L_X=\pi_\ast(\Omega X)\otimes \mathbb{Q}$ is the {\it homotopy Lie algebra} of $X$. Suppose $(\Lambda V_X, d)$ is a {\it Sullivan model} of $X$. The differential $d=\sum\limits_{i\geq 0}d_i$ with $d_i: V_X\rightarrow \Lambda^{i+1} V_X$, and $(\Lambda V_X, d)$ is {\it minimal} if the linear part $d_0=0$. In the latter case, $V_X$ is dual to $\pi_\ast(\Omega X)\otimes \mathbb{Q}$. Moreover, $X$ is coformal if and only if it has a {\it purely quadratic} Sullivan model $C^\ast (L_X, 0)=(\Lambda (sL_X)^{\#}, d_1)$, where $C^\ast(-)$ is the {\it commutative cochain algebra functor}, $s$ is the suspension and $\#$ is the dual operation.

Let $f: X\rightarrow Y$ be a map of simply connected topological spaces. It admits a {\it relative minimal Sullivan model}
\[
\widehat{f}: (\Lambda V_Y,d) \stackrel{}{\longrightarrow} (\Lambda V_Y\otimes \Lambda W,d)
\]
as a special {\it Sullivan representative}, where $(\Lambda V_Y,d)$ is a minimal Sullivan model of $Y$, and $d(W)\subseteq (\Lambda^{+} V_Y\otimes \Lambda W)\oplus \Lambda^{\geq 2} W$. In general, $(\Lambda V_Y\otimes \Lambda W,d)$ itself may be not minimal.
\subsection{Two preliminary lemmas}
\label{subsec: 2lemmas}
\begin{lemma}\label{TNHZlemma}
Let $f: X\rightarrow Y$ be a map of simply connected topological spaces. Let $\widehat{f}: (\Lambda V_Y,d) \stackrel{}{\rightarrow} (\Lambda V_Y\otimes \Lambda W,d)$ a relative minimal Sullivan model of $f$.
Then $f_\ast:\pi_\ast(X)\otimes \mathbb{Q}\rightarrow \pi_\ast(Y)\otimes\mathbb{Q}$ is an epimorphism if and only if $(\Lambda V_Y\otimes \Lambda W,d)$ is a minimal Sullivan algebra.
\end{lemma}
\begin{proof}
The quotient $(\Lambda W, \bar{d})$ of the relative minimal Sullivan model $\widehat{f}$ is a minimal Sullivan model of the homotopy fiber $F$ of $f$, and there is the extension of Sullivan algebras
\[
(\Lambda V_Y,d) \stackrel{\widehat{f}}{\longrightarrow} (\Lambda V_Y\otimes \Lambda W,d)\stackrel{}{\longrightarrow} (\Lambda W, \bar{d}).
\]
The linear part of the extension gives a short exact sequence of cochain complexes
\[
0\stackrel{}{\longrightarrow}(V_Y,0)
\stackrel{\widehat{f}_1}{\longrightarrow} (V_Y\oplus W, d_0)
\stackrel{}{\longrightarrow} (W, 0)
\stackrel{}{\longrightarrow}0,
\]
where $H(\widehat{f}_1): H(V_Y,0)\rightarrow H(V_Y\oplus W, d_0)$ is dual to $f_\ast:\pi_\ast(X)\otimes \mathbb{Q}\rightarrow \pi_\ast(Y)\otimes\mathbb{Q}$. Hence $f_\ast$ is an epimorphism if and only if $H(\widehat{f}_1)$ is a monomorphism, if and only if $d_0=0$, if and only if $(\Lambda V_Y\otimes \Lambda W,d)$ is minimal.
\end{proof}

\begin{lemma}\label{fibmodellemma}
Let 
\[
F\stackrel{i}{\hookrightarrow} E\stackrel{p}{\rightarrow} B
\]
be a fibration of simply connected topological spaces of finite type. Suppose the fibration is TNHZ, and $F$ is coformal. Then there exists a model of the fibration 
\begin{equation}\label{modelFEBeq}
(\Lambda V_B, d)\stackrel{\widehat{p}}{\longrightarrow} (\Lambda V_B\otimes C^\ast (L_F), d)\stackrel{\widehat{i}}{\longrightarrow}C^\ast (L_F, 0),
\end{equation}
where $(\Lambda V_B, d)$ is a minimal Sullivan model of $B$, $(L_F, 0)$ is a Lie model of $F$, and $(\Lambda V_B\otimes C^\ast (L_F), d)$ is a minimal model of $E$.
\end{lemma}
\begin{proof}
Since $F$ is coformal, it has a minimal Sullivan model of the form $C^\ast (L_F, 0)=(\Lambda (sL_F)^{\#}, d_1)$ as the associated commutative cochain algebra of $(L_F,0)$ \cite[Example 7 in Chapter 24 (f)]{FHT}. Let $\widehat{p}: (\Lambda V_B, d)\stackrel{}{\rightarrow} (\Lambda V_B\otimes \Lambda W, d^\prime)$ be a relative minimal Sullivan model of $p$, whose quotient $(\Lambda W,\bar{d})$ is then a minimal Sullivan model of $F$. Since minimal Sullivan model is unique up to isomorphism, $(\Lambda W,\bar{d})\cong C^\ast (L_F, 0)$. It can be used to define an isomorphism of relative minimal Sullivan algebras $\varphi: (\Lambda V_B\otimes C^\ast (L_F), d)\stackrel{}{\rightarrow}(\Lambda V_B\otimes \Lambda W, d^\prime)$, where $\varphi$ is identity on $(\Lambda V_B, d)$ and $d$ is induced from $d^\prime$ through $\varphi$. Then we obtain the extension (\ref{modelFEBeq}) of Sullivan algebras as the model of the fibration. Since the fibration is TNHZ, $p_\ast: \pi_\ast(E)\otimes \mathbb{Q}\rightarrow \pi_\ast(B)\otimes \mathbb{Q}$ is an epimorphism. By Lemma \ref{TNHZlemma}, $(\Lambda V_B\otimes C^\ast (L_F), d)$ is a minimal Sullivan algebra.
\end{proof}

\subsection{A simple case}
\label{subsec: coformalfib1}
For a space $X$, we call $X$ of {\it homotopy dimension $n$} if $\pi_m(X)\otimes \mathbb{Q}=0$ for any $m\geq n+1$.
\begin{proposition}\label{coformalprop1}
Let 
\[
F\stackrel{i}{\hookrightarrow} E\stackrel{p}{\rightarrow} B
\]
be a fibration of simply connected topological spaces of finite type. Suppose the fibration is TNHZ, and $F$ and $B$ are coformal. If $F$ is of homotopy dimension $n$ and $B$ is $m$-connected such that $n\leq m+3$, then $E$ is coformal. 
\end{proposition}
\begin{proof}
By Lemma \ref{fibmodellemma}, the fibration has a model of the form
\[
 C^\ast (L_B, 0)\stackrel{\widehat{p}}{\longrightarrow} (C^\ast (L_B)\otimes C^\ast (L_F), d)\stackrel{\widehat{i}}{\longrightarrow}C^\ast (L_F, 0),
\]
where $(L_B,0)$ is a Lie model of $B$ since $B$ is coformal and $(C^\ast (L_B)\otimes C^\ast (L_F), d)$ is a minimal Sullivan algebra. To show $E$ is coformal, it suffices to show the differential $d$ is quadratic on $(sL_F)^\#$.

The assumption $F$ is of homotopy dimension $n$ is equivalent to that $(sL_F)^\#$ concentrates in degrees less than $n+1$, while $B$ is $m$-connected implies that $(sL_B)^\#$ concentrates in degrees larger than $m$. For any $v\in (sL_F)^\#$, we can write $d(v)=d_1(v)+\theta(v)$ such that $\theta(v)\in \Lambda^{+}(sL_B)^\#\otimes  \Lambda^{\geq 2}((sL_B)^\#\oplus (sL_F)^\#)$. It follows that $n+1\geq |v|+1=|\theta(v)|\geq (m+1)+4$, which contradicts the assumption that $n\leq m+3$. Hence $\theta(v)=0$ for any $v\in (sL_F)^\#$, and $d=d_1$ in $(C^\ast (L_B)\otimes C^\ast (L_F), d)$. This shows that $E$ is coformal and the proposition is proved.
\end{proof}

\begin{example}
Let $S^2 \stackrel{i}{\hookrightarrow} E\stackrel{p}{\rightarrow} B$ be a fibration with $B$ simply connected of finite type. Suppose $i_\ast: \pi_3(S^2)\rightarrow \pi_3(E)$ is rationally nontrivial. Then it is easy to see that $i_\ast$ is a monomorphism on rational homotopy groups and the fibration is TNHZ. Since $S^2$ is of homotopy dimension $3$, $E$ is coformal by Proposition \ref{coformalprop1}. The special case when $B$ is a simply connected $4$-manifold was considered in \cite{Hua}.
\end{example}

\section{Coformal limit} 
\label{sec: limit} 
Let $X$ be a simply connected topological space with a minimal Sullivan model $(\Lambda V, d)$. There is the associated minimal Sullivan algebra $(\Lambda V, d_1)$ with purely quadratic differential. Let $X^\prime$ be a geometric realization of $(\Lambda V, d_1)$. It is clear that $X^\prime$ is coformal.

This construction is functorial.
Let $f: X\stackrel{}{\rightarrow} Y$ be a map of simply connected topological spaces of finite type. Let $\widehat{f}: (\Lambda V_Y,d)\rightarrow (\Lambda V_X, d)$ be a Sullivan representative of $f$. The {\it linear part} of $\widehat{f}$ is a linear map $Q\widehat{f}: V_Y\stackrel{}{\rightarrow} V_X$ defined by $\widehat{f} v- Q\widehat{f}v\in \Lambda^{\geq 2} V_X$ for any $v\in V_Y$. It extends to morphism of minimal Sullivan algebras
\[
\widehat{f}_1:=\Lambda Q\widehat{f}: (\Lambda V_Y,d_1)\rightarrow (\Lambda V_X, d_1),
\]
which can be realized by a rational map $f^\prime: X^\prime \stackrel{}{\rightarrow} Y^\prime$.

\begin{lemma}\label{coformallimitlemma}
The rational homotopy type of $X^\prime$ is independent of the choice of the minimal Sullivan model of $X$.
\end{lemma}
\begin{proof}
Let $(\Lambda V^\prime, d)$ be another minimal Sullivan model of $X$. Since minimal model is unique up to isomorphism, there is an isomorphism of minimal Sullivan algebras
\[
\varphi: (\Lambda V^\prime, d)\stackrel{}{\longrightarrow} (\Lambda V, d).
\]
By the previous argument, there are the induced morphisms 
\[
\varphi_1: (\Lambda V^\prime, d_1)\stackrel{}{\longrightarrow} (\Lambda V, d_1), \ \ 
\varphi^{-1}_1: (\Lambda V, d_1)\stackrel{}{\longrightarrow} (\Lambda V^\prime, d_1),
\]
such that $\varphi_1\circ \varphi_1^{-1}$ and $\varphi_1^{-1}\circ \varphi_1$ are identities. Hence the geometric realization of $(\Lambda V^\prime, d_1)$ is rational homotopy equivalent to that of $(\Lambda V^\prime, d)$. The lemma is proved.
\end{proof}

Based on Lemma \ref{coformallimitlemma}, we call $X^\prime$ the {\it coformal limit} of $X$. By the previous argument, the coformal limit is a covariant functor.
\begin{lemma}\label{coformal=limitlemma}
A simply connected topological space $X$ is coformal if and only if it is rational homotopy equivalent to its coformal limit.
\end{lemma}
\begin{proof}
Let $X^\prime$ be the coformal limit of $X$ as above. If $X\simeq_{\mathbb{Q}} X^\prime$, $X$ is coformal. Conversely, if $X$ is coformal it has a purely quadratic Sullivan model $(\Lambda V, d_1)$. By the definition of coformal limit and Lemma \ref{coformallimitlemma}, the coformal limit $X^\prime$ is a geometric realization of $(\Lambda V, d_1)$. Hence $X\simeq_{\mathbb{Q}} X^\prime$ and the lemma is proved.
\end{proof}

From the Lie point of view, let $(L, d_L)$ be a Lie model of $X$. Then $H(L)$ is isomorphic to the homotopy Lie algebra of $X$. And it is easy to see that $(H(L), 0)$ is a Lie model of the coformal limit $X^\prime$. Let $f: X\rightarrow Y$ be a map of simply connected spaces with a {\it Lie representative} $\widetilde{f}: (L_X, d_{L_X})\stackrel{}{\rightarrow} (L_Y, d_{L_Y})$. Then there is the induced morphism $H(f): (H(L_X), 0)\rightarrow (H(L_Y), 0)$. Since the linear part of a morphism between minimal Sullivan algebras is dual to the Lie algebra morphism between their homotopy Lie algebras, $H(f)$ is a Lie representative of $f^\prime: X^\prime \rightarrow Y^\prime$.

\section{Proof of of main theorems} 
\label{sec: proofmainthm} 
In this section, we prove the main theorems, Theorem \ref{coformalmainthm}, Theorem \ref{coformalmainthm2}, and Theorem \ref{fibkoszulthm2}, of this paper.

\subsection{Proof of Theorem \ref{coformalmainthm}}
\label{subsec: proofthm1}
Let 
\begin{equation}\label{Efibeq}
F\stackrel{i}{\hookrightarrow} E\stackrel{p}{\rightarrow} B
\end{equation}
be a fibration of simply connected topological spaces of finite type. Suppose the fibration is TNHZ, and $F$ and $B$ are coformal. By Lemma \ref{fibmodellemma}, the fibration has a model of the form
\begin{equation}\label{modelFEBeq2}
 C^\ast (L_B, 0)\stackrel{\widehat{p}}{\longrightarrow} (C^\ast (L_B)\otimes C^\ast (L_F), d)\stackrel{\widehat{i}}{\longrightarrow}C^\ast (L_F, 0),
\end{equation}
where $(L_B,0)$ is a Lie model of $B$ and $(C^\ast (L_B)\otimes C^\ast (L_F), d)$ is a minimal Sullivan algebra. It induces an extension of minimal Sullivan algebras with purely quadratic differentials
\begin{equation}\label{indud1exteq}
 C^\ast (L_B, 0)\stackrel{\widehat{p}_1}{\longrightarrow} (C^\ast (L_B)\otimes C^\ast (L_F), d_1)\stackrel{\widehat{i}_1}{\longrightarrow}C^\ast (L_F, 0).
\end{equation}
It can be realized as a rational homotopy fibration
\begin{equation}\label{E'fibeq}
F\stackrel{}{\rightarrow} E^\prime\stackrel{}{\rightarrow} B,
\end{equation}
where $E^\prime$ is a geometric realization of $(C^\ast (L_B)\otimes C^\ast (L_F), d_1)$. In particular, $E^\prime$ is coformal and is the coformal limit of $E$.

Recall for a simply connected space $X$ with a minimal Sullivan algebra $(\Lambda V,d)$, there is the so-called {\it rational Toomer invariant} $e_0(X)$, which is equal to its algebraic counterpart, the {\it Toomer invariant} $e(\Lambda V,d)$, defined as the least integer $r$ such that the natural quotient
\[
\rho_r: (\Lambda V,d)\stackrel{}{\longrightarrow} (\Lambda V/ \Lambda^{>r}V, d)
\]
is injective on cohomology.

\begin{proposition}\label{coformalprop2}
Let 
\[
F\stackrel{i}{\hookrightarrow} E\stackrel{p}{\rightarrow} B
\]
be a fibration of simply connected topological spaces of finite type. Suppose the fibration is TNHZ, and $F$ and $B$ are coformal. If $e_0(E^\prime)\leq 2$, then $E$ is coformal. 
\end{proposition}
\begin{proof}
By Lemma \ref{fibmodellemma}, the fibration has a model of the form (\ref{modelFEBeq2}) where $(C^\ast (L_B)\otimes C^\ast (L_F), d)$ is a minimal model of $E$. To show $E$ is coformal, it suffices to show that $(C^\ast (L_B)\otimes C^\ast (L_F), d)$ is isomorphic to a purely quadratic minimal Sullivan algebra.

We prove this by induction. Since $d=d_1$ on $C^\ast(L_B)$, we only need to consider the differential $d$ on $(sL_F)^{\#}$.
Let $\{v_\alpha\}_{\alpha}$ be a basis of $(sL_F)^{\#}$. Suppose $\{v_{\alpha_1},\ldots, v_{\alpha_k}\}\subseteq \{v_\alpha\}_{\alpha}$ be the set of the generators with lowest degree such that
\[
d(v)=d_1(v)+\theta(v) ~{\rm with}~0\neq\theta(v)\in \Lambda^{+}(sL_B)^\#\otimes  \Lambda^{\geq 2}((sL_B)^\#\oplus (sL_F)^\#),
\] 
for any $v\in \{v_{\alpha_1},\ldots, v_{\alpha_k}\}$.
Then by degree reason, $d_1\theta(v_{\alpha_i})=d_1(d-d_1)(v_{\alpha_i})=d_1d(v_{\alpha_i})=dd(v_{\alpha_i})=0$. Hence, each $\theta(v_{\alpha_i})$ is a cocycle of wordlength larger than $2$ in $(C^\ast (L_B)\otimes C^\ast (L_F), d_1)$. The latter quadratic Sullivan algebra is a minimal model of the coformal limit $E^\prime$. Since $e_0(E^\prime)\leq 2$, $\theta(v_{\alpha_i})$ is coboundary and there exists an element $z_{\alpha_i}\in C^\ast (L_B)\otimes C^\ast(L_F)$ such that $d_1(z_{\alpha_i})=\theta(v_{\alpha_i})$. Note that the wordlength of $z_{\alpha_i}$ is larger than $1$, and hence $d(z_{\alpha_i})=d_1(z_{\alpha_i})=\theta(v_{\alpha_i})$ by degree reason. We then have
$d(v_{\alpha_i})=d_1(v_{\alpha_i})+d(z_{\alpha_i})$. 
Let $V$ be the graded vector space spanned by $\{v_{\alpha_1}^\prime, \ldots v_{\alpha_k}^\prime\}\cup \{v_\alpha\}_{\alpha\neq \alpha_1, \ldots \alpha_k}$ with $|v_{\alpha_i}^\prime|=|v_{\alpha_i}|$. There is an isomorphism of minimal Sullivan algebras
\[
\varphi: (C^\ast (L_B)\otimes \Lambda V, d^\prime)\stackrel{}{\longrightarrow}(C^\ast (L_B)\otimes C^\ast (L_F), d),
\]
where $\varphi$ is identity on $(C^\ast (L_B),d)$ and $\{v_\alpha\}_{\alpha\neq \alpha_1, \ldots \alpha_k}$, $\varphi(v^\prime_{\alpha_i})=v_{\alpha_i}-z_{\alpha_i}$, and $d^\prime=\varphi^{-1}\circ d\circ\varphi$. 
Then $d^\prime(v^\prime_{\alpha_i})=(\varphi^{-1}\circ d)( v_{\alpha_i}-z_{\alpha_i})=\varphi^{-1}(d_1(v_{\alpha_i}))=d_1(v_{\alpha_i})$. 
Hence we have constructed a minimal Sullivan model $(C^\ast (L_B)\otimes \Lambda V, d^\prime)$ of $E$ such that $d^\prime$ is quadratic for elements of degree less or equal to $|v_{\alpha_1}|$. Then by induction on the degree and repeating the above argument, we finally can obtain a minimal Sullivan model of $E$ with purely quadratic differential. This shows that $E$ is coformal and the proposition is proved.
\end{proof}

We can now prove Theorem \ref{coformalmainthm}.

\begin{proof}[Proof of Theorem \ref{coformalmainthm}]
It is known that coformal spaces \cite[p. 30]{FH} satisfy $cat_0=e_0$. Theorem \ref{coformalmainthm} follows from Proposition \ref{coformalprop2} immediately.
\end{proof}

\subsection{Proof of Theorem \ref{coformalmainthm2}}
\label{subsec: proofthm2}

Let 
\begin{equation}\label{Ycofibeq}
\Sigma A\stackrel{f}{\longrightarrow} Y \stackrel{h}{\longrightarrow} Z
\end{equation}
be a homotopy cofibration of simply connected topological spaces of finite type. 
The map $f$ is {\it inert} if $h$ is surjective in rational homotopy. Notice that when $A=S^{k}$ for some $k$, this is the classical notion of inert map. Theorem \ref{coformalmainthm2} concerns the relation between the coformality of $Y$ and that of $Z$. To prove it, we need a preliminary result from~\cite{BT}. 
\begin{theorem}[rational version of Proposition 3.5 in~\cite{BT}] 
\label{BTtheorem} 
Suppose that 
$\Sigma A\stackrel{f}{\rightarrow} Y \stackrel{h}{\rightarrow} Z$
is a homotopy cofibration of simply-connected spaces of finite type and $f$ is inert.
Then there is a rational homotopy fibration 
\[
(\Omega Z\wedge\Sigma A)\vee\Sigma A \stackrel{}{\longrightarrow} Y \stackrel{h}{\longrightarrow}Z,
\]
which splits after looping to give a rational homotopy equivalence 
\[
 \hspace{3.75cm}
 \Omega Y\simeq_{\mathbb{Q}}\Omega Z\times\Omega((\Omega Z\wedge\Sigma A)\vee\Sigma A).
 \hspace{3.75cm}\Box\] 
\end{theorem} 

We are ready to prove Theorem \ref{coformalmainthm2}.
Recall for the space $Y$, we have defined its coformal limit $Y^\prime$.

\begin{proof}[Proof of Theorem \ref{coformalmainthm2}]
By Theorem \ref{BTtheorem}, there is a rational homotopy fibration 
\begin{equation}\label{BTfibeq}
(\Omega Z\wedge\Sigma A)\vee\Sigma A \stackrel{}{\longrightarrow} Y \stackrel{h}{\longrightarrow}Z.
\end{equation}
Since suspension is wedge of spheres rationally, $(\Omega Z\wedge\Sigma A)\vee\Sigma A$ is rational homotopy equivalent to a wedge of spheres, hence is coformal. On the other hand, $f$ is inert, which is equivalent to that the rational fibration (\ref{BTfibeq}) is TNHZ. Applying Theorem \ref{coformalmainthm} to (\ref{BTfibeq}), $Y$ is coformal by the assumption.
\end{proof}

\begin{example}\label{coformalfibex4}
Let $Y$ be a connected sum of products of two simply connected spheres: $Y=(S^{k_1}\times S^{n-k_1})\#\cdots \# (S^{k_r}\times S^{n-k_r})$. This is a classical example of coformal spaces, which can be understood through Theorem \ref{coformalmainthm2}. Indeed, there is the homotopy cofibration
\[
(S^{k_2}\vee S^{n-k_2})\vee\cdots \vee (S^{k_r}\vee S^{n-k_r}) \stackrel{f}{\longrightarrow} Y \stackrel{h}{\longrightarrow} S^{k_1}\times S^{n-k_1},
\]
where $S^{k_1}\times S^{n-k_1}$ is coformal and $cat_0(Y^\prime)=cat_0(Y)=2$. 
Moreover, there is the homotopy commutative digram
\[
\begin{gathered}
\xymatrix{
\mathop{\bigvee}\limits_{i=1}^{r}(S^{k_i}\vee S^{n-k_i}) \ar[r]^<<<<<{q} \ar[d]^{i} &
S^{k_1}\vee S^{n-k_1} \ar[d]^{j} \\
Y \ar[r]^<<<<<<<<<<{h} &
S^{k_1}\times S^{n-k_1},
}
\end{gathered}
\]
where $q$ is the obvious projection and $i$ and $j$ are the inclusions of lower skeletons. Since $q_\ast$ and $j_\ast$ are surjective on homotopy groups, so is $h_\ast$. Hence $f$ inert, and $Y$ satisfies all the conditions of Theorem \ref{coformalmainthm2}.
\end{example}

\subsection{Proof of Theorem \ref{fibkoszulthm2}}
\label{subsec: proofthm3}
In the end of the section, let us prove Theorem \ref{fibkoszulthm2} from Proposition \ref{fibkoszulthm}.

\begin{proof}[Proof of Theorem \ref{fibkoszulthm2}]
First it is known that $S^n$ and $B$ are both Koszul, and $S^n$ is elliptic. We prove the theorem case by case.

Case (1).~When $n$ is odd, $cat_0(E^\prime)\leq cat_0(B)+cat_0(S^n)=2$ by Example \ref{coformalfibex1intro}. Further, $i^\ast$ is rationally nontrivial on cohomology implies that the fibration is both TNCZ and TNHZ. Then $E$ is Koszul by Proposition \ref{fibkoszulthm}.

Case (2).~When $n$ is even, $S^n$ is a so-called {\em $F_0$-space}, that is, an elliptic space with positive Euler characteristic. By Lemma \ref{fibmodellemma}, the fibration has a model of the form
\[
(\Lambda V_B, d_1)\stackrel{}{\longrightarrow} (\Lambda V_B\otimes \Lambda (a, b), d)\stackrel{\widehat{i}}{\longrightarrow} (\Lambda (a, b), db=a^2),
\]
where $|a|=n$, $(\Lambda V_B, d_1)$ and $(\Lambda V_B\otimes \Lambda (a, b), d)$ are minimal Sullivan models of $B$ and $E$ respectively. Since the fibration is TNCZ, $da=0$ in $(\Lambda V_B\otimes \Lambda (a, b), d)$
Consider the induced fibration $F\rightarrow E^\prime\rightarrow B$ of coformal limits. 
It has a model of the form
\[
(\Lambda V_B, d_1)\stackrel{}{\longrightarrow} (\Lambda V_B\otimes \Lambda (a, b), d_1)\stackrel{}{\longrightarrow} (\Lambda (a, b), db=a^2).
\]
In particular, $d_1a=0$, $d_1b=a^2-\theta(b)$ such that $\theta(b)\in \Lambda ^{+}V_B\otimes \mathbb{Q}\{a\}$ has homogenous length $2$, and the induced fibration is also TNCZ. Hence, $d_1(\theta(b))=0$, and then $[a]^2=[\theta(b)]$ in the algebra $H^\ast (\Lambda V_B\otimes \Lambda (a, b), d_1)$. It is clear that the {\it cup length} of $B$ is $1$ as it is a wedge of spheres.

{\it Claim $(A)$: $0=[a]\cup [\theta(b)]\in H^\ast(E^\prime;\mathbb{Q})$.}

First, notice that $\theta(b)=0$ or $\theta(b)=a y$ for some $y\in V_B$, or $\theta (b)\in \Lambda^2V_B$.
If $\theta(b)=a y$, $|y|=n$ is even. Then since the original fibration is TNCZ and $B$ is a wedge of spheres, $dy=0$ in $(\Lambda V_B\otimes \Lambda (a, b), d)$. In particular, $d_1(y)=0$ in $(\Lambda V_B\otimes \Lambda (a, b), d_1)$. It implies that as cohomology classes of $H^\ast(\Lambda V_B\otimes \Lambda (a, b), d_1)$
\[
[a]\cup [\theta(b)]=[a]\cup [a] \cup [y]=[\theta(b)] \cup [y]= [a] \cup [y]\cup [y]=0.
\]
If $\theta (b)\in \Lambda^2V_B$, then $[\theta(b)]=0$ in $H^\ast(\Lambda V_B\otimes \Lambda (a, b), d_1)$. Hence in all cases, $[a]\cup [\theta(b)]=0$ in $H^\ast(\Lambda V_B\otimes \Lambda (a, b), d_1)\cong H^\ast(E^\prime;\mathbb{Q})$, and the Claim ($A$) is proved.

{\it Claim $(B)$: the cup length of $E^\prime$ is at most $2$.}

Since the induced fibration is TNCZ, $H^\ast(E^\prime;\mathbb{Q})\cong H^\ast(B;\mathbb{Q})\otimes \mathbb{Q}\{[a]\}$ as $H^\ast(B;\mathbb{Q})$-module. For any three elements $k_i [a]+ l_i x_i$ with $x_i\in H^+(B;\mathbb{Q})$ and $k_i$, $l_i\in \mathbb{Q}$ ($i=1, 2, 3$), by Claim $(A)$ and the fact that $B$ has cup length $1$
\[
\begin{split}
&\ \ \ \ (k_1 [a]+ l_1 x_1)\cup (k_2 [a]+ l_2 x_2) \cup (k_3 [a]+ l_3 x_3)\\
&= (k_1k_2 [\theta(b)]+ k_1 l_2[a] \cup x_2+k_2 l_1[a]\cup x_1)\cup (k_3 [a]+ l_3 x_3)\\
&=k_1k_2k_3 [a]\cup  [\theta(b)]+  k_1 l_2 k_3  [\theta(b)] \cup x_2 + k_2 l_1 k_3  [\theta(b)]\cup x_1\\
&=0.\\
\end{split}
\]
By similar computation, the cup product of any other three elements in $H^\ast(E^\prime;\mathbb{Q})$ vanishes.
This means that $E^\prime$ has cup length at most $2$, and the Claim ($B$) is proved.

Lastly by \cite[Proposition 3.2]{Lup} or \cite[Corollary B]{AK} $E^\prime$ is formal, and hence its rational category is equal to its cup length \cite[Chapter 29, Example 4]{FHT}. From Claim ($B$) we have $cat_0(E^\prime)\leq 2$. Then $E$ is Koszul by Proposition \ref{fibkoszulthm}.

Case (3).~When $n$ is even and $B$ is a wedge of odd dimensional spheres, the fibration is TNCZ implies that it is rationally trivial and $E\simeq_{\mathbb{Q}} B\times F$ by \cite[Theorem 2.2]{Lup}. In particular, $E$ is Koszul. 
\end{proof}

\section{Discussions on the conditions of Theorem \ref{coformalmainthm}}
\label{sec: counterexamples}
There are several conditions imposed in Theorem \ref{coformalmainthm}. One may wonder whether these conditions are necessary or can be weakened. The answer is negative. First, in Proposition \ref{coformalprop3} we show that the coformality of the total space implies the coformality of the fiber in a fibration. Thus the condition that the fiber is coformal is natural. Moreover, we construct Example \ref{counterex1} to show that the TNHZ condition is necessary, and Example \ref{counterex3} to show that the category inequality condition is optimal. Finally, through Example \ref{counterex2} we also prove that the converse statement of Theorem \ref{coformalmainthm} is false, that is, under the same conditions the coformality of the total space does not necessarily imply the coformality of the base.

\begin{proposition}\label{coformalprop3}
Let 
\[
F\stackrel{i}{\hookrightarrow} E\stackrel{p}{\rightarrow} B
\]
be a fibration of simply connected topological spaces of finite type. Suppose the fibration is TNHZ. If $E$ is coformal, then $F$ is coformal.
\end{proposition}
\begin{proof}
Let $(\Lambda V_B, d)$ be a minimal model of $B$.
By the similar argument in Lemma \ref{fibmodellemma}, the fibration has a model of the form 
\[
(\Lambda V_B, d)\stackrel{\widehat{p}}{\longrightarrow} (\Lambda V_B\otimes \Lambda V_F, d)\stackrel{\widehat{i}}{\longrightarrow}(\Lambda V_F, \bar{d}),
\]
where $(\Lambda V_B\otimes \Lambda V_F, d)$ is a minimal model of $E$. To show $F$ is coformal, it suffices to show that $(\Lambda V_F, \bar{d})$ is isomorphic to a purely quadratic Sullivan algebra.

Since $E$ is coformal, it has a minimal Sullivan model $C^\ast (L_E, 0)$. By the uniqueness of minimal Sullivan model, there is an isomorphism of minimal Sullivan algebras $
\varphi: (\Lambda V_B\otimes \Lambda V_F, d)\stackrel{}{\rightarrow} C^\ast (L_E, 0)$, which induces an isomorphism of minimal Sullivan algebras $\varphi_1: (\Lambda V_B\otimes \Lambda V_F, d_1)\stackrel{}{\rightarrow} C^\ast (L_E, 0)$, where $(\Lambda V_B, d_1)$ is a Sullivan subalgebra of $(\Lambda V_B\otimes \Lambda V_F, d_1)$. Hence, there is an isomorphism of minimal Sullivan algebras 
\[
\psi:=\varphi_1^{-1}\circ \varphi:  (\Lambda V_B\otimes \Lambda V_F, d)\stackrel{}{\longrightarrow}  (\Lambda V_B\otimes \Lambda V_F, d_1),
\]
whose linear part $Q\psi$ is the identity. Then it is clear that $\psi$ sends the differentiable ideal $\Lambda^{+} V_B\otimes \Lambda V_F$ isomorphically onto the ideal $\Lambda^{+} V_B\otimes \Lambda V_F$ of $(\Lambda V_B\otimes \Lambda V_F, d_1)$, and it induces an isomorphism $\bar{\psi}: (\Lambda V_F, \bar{d})\stackrel{}{\longrightarrow} (\Lambda V_F, \bar{d}_1)$ whose linear part is identity. Hence $F$ is coformal and the proposition is proved.
\end{proof}

\begin{example}[TNHZ is necessary]\label{counterex1}
Consider the fiber bundle 
\[
S^2\stackrel{}{\longrightarrow} \mathbb{C}P^3\stackrel{}{\longrightarrow} S^4
\]
defined in \cite[Section 1.1]{HBJ}. $S^2$ and $S^4$ are coformal. The coformal limit of $\mathbb{C}P^3$ is $S^2\times S^7$, and hence has rational category equal to $2$. However, $\mathbb{C}P^3$ is not coformal, and the fiber bundle is not TNHZ. This examples shows that TNHZ condition is necessary in Theorem \ref{coformalmainthm}.
\end{example}

\begin{example}[Category condition is optimal]\label{counterex3}
Consider the rational fibration $S^9\stackrel{i}{\longrightarrow}E\stackrel{p}{\longrightarrow} S^2\times S^2$ determined the extension of minimal Sullivan algebras
\[
\begin{split}
&(\Lambda (x_2, y_3, a_2, b_3), dy_3=x_2^2, db_3=a_2^2) \\
 & \stackrel{\widehat{p}}{\longrightarrow}(\Lambda (x_2, y_3, a_2, b_3, s_9), dy_3=x_2^2, db_3=a_2^2, ds_9=a_2x_2b_3y_3)  \stackrel{\widehat{i}}{\longrightarrow}  (\Lambda s_9, 0),
\end{split}
\]
where the subscripts indicate the degree of the generators. We claim that $E$ is not coformal. Indeed suppose otherwise $E$ is coformal. Then $E$ is rational homotopy equivalent to its coformal limit $E^\prime \simeq_{\mathbb{Q}} S^2\times S^2\times S^9$. It implies that there is an isomorphism of minimal Sullivan algebras
\[
\begin{split}
&\varphi:(\Lambda (x_2, y_3, a_2, b_3, s_9), dy_3=x_2^2, db_3=a_2^2, ds_9=a_2x_2b_3y_3)\\
&\stackrel{}{\longrightarrow}(\Lambda (\widetilde{x}_2, \widetilde{y}_3, \widetilde{a}_2, \widetilde{b}_3, \widetilde{s}_9), d\widetilde{y}_3=\widetilde{x}_2^2, d\widetilde{b}_3=\widetilde{a}_2^2, d\widetilde{s}_9=0).
\end{split}
\]
By degree reason, $\varphi(s_9)=k_0\widetilde{s}_9+ \widetilde{x}_2^2\widetilde{a}_2(k_1\widetilde{y}_3+k_2\widetilde{b}_3)+ \widetilde{a}_2^2\widetilde{x}_2(k_3\widetilde{y}_3+k_4\widetilde{b}_3)$ for some $k_i\in \mathbb{Q}$. Applying the differential, we have $\varphi(a_2x_2b_3y_3)=\widetilde{x}_2^2\widetilde{a}_2(k_1\widetilde{x}_2^2+k_2\widetilde{a}_2^2)+ \widetilde{a}_2^2\widetilde{x}_2(k_3\widetilde{x}_2^2+k_4\widetilde{a}_2^2)$, which is impossible. Hence $E$ is not coformal with $E^\prime \simeq_{\mathbb{Q}} S^2\times S^2\times S^9$. In particular, $cat_0(E^\prime)=3$. However the fiber $S^2$ and the base $S^2\times S^2$ are coformal and the fibration is TNHZ. This shows that the condition $cat_0(E^\prime)\leq 2$ in Theorem \ref{coformalmainthm} is optimal.
\end{example}

\begin{example}[Converse statement is false]\label{counterex2}
Consider the rational fibration $S^3\stackrel{i}{\longrightarrow}E\stackrel{p}{\longrightarrow} \mathbb{C}P^2$ determined the extension of minimal Sullivan algebras
\[
(\Lambda (x_2, y_5), dy_5=x_2^3) \stackrel{\widehat{p}}{\longrightarrow} (\Lambda (x_2, y_5, a_3), dy_5=x_2^3, da_3=x_2^2)  \stackrel{\widehat{i}}{\longrightarrow}  (\Lambda a_3, 0),
\]
where the subscripts indicate the degree of the generators. We claim that $E$ is rational homotopy equivalent to $S^2\times S^5$. Indeed, in the model $(\Lambda (x_2, y_5, a_3), dy_5=x_2^3, da_3=x_2^2)$ we see that $d(y_5-x_2a_3)=0$. Then we can define an isomorphism of minimal Sullivan algebras
\[
\varphi: (\Lambda (x_2, y_5, a_3), dy_5=x_2^3, da_3=x_2^2) \stackrel{}{\longrightarrow}(\Lambda (\widetilde{x}_2, \widetilde{y}_5, \widetilde{a}_3), d\widetilde{y}_5=0, d\widetilde{a}_3=\widetilde{x}_2^2), 
\]
by $\varphi(x_2)=\widetilde{x}_2$, $\varphi(a_3)=\widetilde{a}_3$, and $\varphi(y_5)=\widetilde{y}_5+\widetilde{x}_2\widetilde{a}_3$. Since the geometric realization of $(\Lambda (\widetilde{x}_2, \widetilde{y}_5, \widetilde{a}_3), d\widetilde{y}_5=0, d\widetilde{a}_3=\widetilde{x}_2^2)$ is $S^2\times S^5$, we have showed the claim that $E\simeq_{\mathbb{Q}} S^2\times S^5$. In particular $E$ is coformal and $cat_0(E)=2$. It is also clear that the fibration is TNHZ and the fiber $S^3$ is coformal. However, the base $\mathbb{C}P^2$ is not coformal. This example shows that the coformality of the total space does not imply the coformality of the base space.
\end{example}


\bibliographystyle{amsalpha}

\end{document}